\documentclass[12pt,a4]{amsart}
\usepackage{amsmath,times,epsfig,amssymb,amsbsy,amscd,amsfonts,amstext,color,bm,mathrsfs}
 
\usepackage[margin=1in]{geometry}
\usepackage{enumerate}
\usepackage{placeins}
\usepackage{array}
\usepackage{tabulary}
\usepackage{multirow}
\usepackage{graphicx}
\usepackage{hyperref}
\usepackage{hhline}
\usepackage{soul}
\usepackage{multirow}
\usepackage[table, svgnames, dvipsnames]{xcolor}
\usepackage{makecell, cellspace, caption}
\usepackage{subcaption}
 
\usepackage[color,matrix,arrow]{xypic}
\usepackage{tikz}
        \usetikzlibrary{arrows.meta}
        \usetikzlibrary{arrows}
        \usetikzlibrary{quotes}
        \usetikzlibrary{graphs}
        \usetikzlibrary{positioning}
\usepackage{tikz-cd}
\hypersetup{
  colorlinks   = true, 
  urlcolor     = blue, 
  linkcolor    = blue, 
  citecolor   = red  
}
\usetikzlibrary{decorations}
\usetikzlibrary{arrows}
\tikzstyle{v} = [circle, draw, inner sep=2pt, minimum size=3pt, fill=black]
\tikzstyle{l} = [rectangle, draw, rounded corners]

\theoremstyle{plain}
\newtheorem{theorem}{Theorem}[section]
\newtheorem{corollary}[theorem]{Corollary}
\newtheorem{lemma}[theorem]{Lemma}

\newtheorem{proposition}[theorem]{Proposition}

\makeatletter
\@namedef{subjclassname@2020}{%
  \textup{2020} Mathematics Subject Classification}
\makeatother

\theoremstyle{definition}
\newtheorem{definition}[theorem]{Definition}
\newtheorem{example}[theorem]{Example}

\newtheorem{notation}[theorem]{Notation}
\theoremstyle{remark}
\newtheorem{remark}[theorem]{Remark}

\newcommand{\A}{\mathcal{A}}
\newcommand{\As}{\mathscr{A}}
\newcommand{\B}{\mathsf{B}}
\newcommand{\C}{\mathbb{C}}

\newcommand{\Cca}{\mathcal{C}}

\newcommand{\Q}{\mathbb{Q}}
\newcommand{\R}{\mathbb{R}}

\newcommand{\scR}{\mathcal{R}}

\newcommand{\Z}{\mathbb{Z}}

\newcommand{\E}{{\mathcal{E}}}

\newcommand{\vn}{\noindent}
\newcommand{\asc}{{\rm asc}}
\newcommand{\dsc}{{\rm dsc}}

\newcommand{\M}{\mathcal{M}}
\newcommand{\tbf}{\textbf} 
\newcommand{\cc}{\textbf{c}}

\newcommand{\Shi}{\mathrm{Shi}}
\newcommand{\Cat}{\mathrm{Cat}}
\newcommand{\Lin}{\mathrm{Lin}}
\newcommand{\Ish}{\mathrm{Ish}}

\newcommand{\lcm}{\operatorname{lcm}}

\newcommand{\quasi}{\operatorname{quasi}}
\newcommand{\arith}{\operatorname{arith}}

\newcommand{\Mat}{\operatorname{Mat}}
 
\newcommand{\sign}{\operatorname{sign}} 
 
\newcommand{\Part}{\operatorname{Part}} 

\newcolumntype{K}[1]{>{\centering\arraybackslash}p{#1}}

\begin{document}

\title[Signatures of Type $A$ Root Systems]{Signatures of Type $A$ Root Systems}

\begin{abstract}\label{sec:intro}
Given a type $A$ root system $\Phi$ of rank $n$, we introduce the concept of a signature for each subset $S$ of $\Phi$ consisting of $n+1$ positive roots. For a subset $S$ represented by a tuple $(\beta_1, \ldots, \beta_{n+1})$, the signature of $S$ is defined as an unordered pair $\{a, b\}$, where $a$ and $b$ denote the numbers of $1$s and $-1$s, respectively, among the cofactors $(-1)^k \det(S \setminus \{\beta_k\})$ for $1 \le k \le n+1$. We prove that the number of tuples with a given signature can be expressed in terms of classical Eulerian numbers. The study of these signatures is motivated by their connections to the arithmetic and combinatorial properties of cones over deformed arrangements defined by $\Phi$, including the Shi, Catalan, Linial, and Ish arrangements. We apply our main result to compute two important invariants of these arrangements: The minimum period of the characteristic quasi-polynomial, and the evaluation of the classical and arithmetic Tutte polynomials at $(1, 1)$.
\end{abstract}

\author{Michael Cuntz}
\address{Michael Cuntz, Institut f\"ur Algebra, Zahlentheorie und Diskrete Mathematik, Fakult\"at f\"ur Mathematik und Physik, Leibniz Universit\"at Hannover, Welfengarten 1, D-30167 Hannover, Germany}
\email{cuntz@math.uni-hannover.de}
 
\author{Hung Manh Tran}
\address{Hung Manh Tran, Faculty of Fundamental Sciences, Phenikaa University, Hanoi 12116, Vietnam.}
\email{hung.tranmanh@phenikaa-uni.edu.vn}

\author{Tan Nhat Tran}
\address{Tan Nhat Tran, Department of Mathematics and Statistics, Binghamton University (SUNY), Binghamton, NY 13902-6000, USA.}
\email{tnhattran@binghamton.edu}

\author{Shuhei Tsujie}
\address{Shuhei Tsujie, Department of Mathematics, Hokkaido University of Education, Asahikawa, Hokkaido 070-8621, Japan}
\email{tsujie.shuhei@a.hokkyodai.ac.jp}


\date{\today}
\maketitle


\section{Introduction}
The primary aim of this paper is to explore the arithmetic and combinatorial properties of a root system of type $A$. 
For $n \in \Z_{\ge 0}$, denote $[n]:=\{1,2,\ldots,n\}$.
Let $\E=\{\epsilon_1, \ldots, \epsilon_{n+1}\}$ be an orthonormal basis for $V= \R^{n+1}$. 
Define  a subspace $U$ of $V$ as follows:
$$U : = \left\{ \sum_{i=1}^{n+1} r_i\epsilon_i  \in V \,\middle\vert\, \sum_{i=1}^{n+1} r_i=0 \right\} \simeq \R^{n}.$$ 

The root system $\Phi$   of type $A_{n}$ is given by
 $$\Phi = \{\pm(\epsilon_i - \epsilon_j) \mid 1 \le i<j \le n+1\},$$ 
which lies in $U$, with a set of simple roots 
 $$\Delta = \{\alpha_i:=\epsilon_i - \epsilon_{i+1} \mid 1 \le i  \le n\},$$
 and the associated positive system
  $$\Phi^+ =  \left\{ \alpha_{i,j} :=\epsilon_i-\epsilon_j =  \sum_{k=i}^{j-1} \alpha_k\,\middle\vert\,1 \le i <  j \le n+1\right\}.$$

We will also describe this root system using matrices.
Let $n,p\in\Z_{>0}$ be positive integers, and let $\Mat_{n\times p}(\Z)$ denote the set of all $n \times p$ matrices with integer entries.
Let $S \subseteq \Phi^+$ with $\#S=r$. 
If we want to emphasize the ordering of elements in $S$, we represent $S$ as a tuple $S=(\beta_1,\ldots,\beta_{r}) \in (\Phi^+)^r$. 
Write  $\beta_j = \sum_{i=1}^nc_{ij}\alpha_i$ for each $1 \le j \le r$.
Denote by $\Cca_S := (\beta_1\,\cdots\,\beta_r)= (c_{ij}) \in \Mat_{n\times r}(\Z)$ the coefficient matrix of $S$ with respect to $\Delta$, where the columns of this matrix correspond to the vectors $\beta_i$.

\begin{remark} 
\label{rem:id} 
We will often identify \( S \) with \( \Cca_S \) when there is no risk of ambiguity. The matrix 
\[
(\epsilon_1 - \epsilon_2 \quad \epsilon_2 - \epsilon_3  \quad \cdots  \quad \epsilon_n - \epsilon_{n+1} \quad \epsilon_{n+1})
\]
serves as the transition matrix from the basis \( \Delta \cup \{\epsilon_{n+1}\} \) to the basis \( \mathcal{E} \) in \( V = \mathbb{R}^{n+1} \). This matrix preserves all key concepts in this paper, including determinant, signature, elementary divisors, lcm period, arithmetic multiplicity, and matroid base. Therefore, we can interchangeably refer to either of the above descriptions of a root system of type \( A \). 
However, unless otherwise specified, we always treat \( \Phi \) as a root system of rank \( n \) in \(U \simeq \mathbb{R}^n \) with simple roots \( \Delta =\{\alpha_1,\ldots,\alpha_n \}\). 
The description of $\Phi$ in terms of the basis $\E=\{\epsilon_1, \ldots, \epsilon_{n+1}\}$ for $V= \R^{n+1}$ will be used in \S\ref{sec:SE} 
when we identify $\Phi^+$ with the complete graph $K_{n+1}$ with vertex set $[n+1]$.
 
\end{remark}

The following fact is well-known.
\begin{lemma}
\label{lem:-1,0,1}
If $S \subseteq \Phi^+$ with $\#S=n$, then $\det(S) \in \{-1,0,1\}$.
\end{lemma}

Now, we introduce the main concept of this paper.

\begin{definition}
\label{def:sign}
Let $\Phi$ be root system of type $A_n$, and let $S=(\beta_1,\ldots,\beta_{n+1})\in(\Phi^+)^{n+1}$. 
The \tbf{signature} $\sign(S)$ of $S$ is defined as an unordered pair $\{a,b\}$, where $a$ and $b$ are the numbers of $1$s and $-1$s, respectively, among the cofactors\footnote{These numbers are referred to as cofactors, inspired by the fact that they are the usual cofactors (up to multiplication by $(-1)^{n+1}$) of the corresponding deformation matrix (see Case \ref{C3} in \S\ref{sec:SB}).}
 $$d_k:=(-1)^k \det(S_k)\quad \text{for}  \quad 1 \le k \le n+1,$$ 
 where $S_k := (\beta_1, \ldots ,\widehat{\beta_{k}}, \ldots ,\beta_{n+1})$, and $\widehat{\beta}$ denotes the omission of $\beta$ from $S$.
\end{definition}
\vn

 \begin{example} 
\label{ex:1st}
Let $\Phi=A_3$ and $S = (\alpha_{1,2}, \alpha_{2,3}, \alpha_{1,4},  \alpha_{2,4}) \subseteq\Phi^+$. 
The coefficient matrix $\Cca_S$ of  $S$ is given by 
$$
\Cca_S = 
\begin{pmatrix}
1 & 0 & 1 &   0 \\
0 & 1 & 1 &   1 \\
0 & 0 & 1 &   1 
\end{pmatrix}.
$$
One may compute  $\sign(S)= \{1,2\}$.

\end{example}

We note that the signature of a subset is independent of the ordering of its elements, as shown in Lemma \ref{lem:unchanged}.

 \begin{remark} 
\label{rem:more}
There are two possible ways to generalize the definition of signature. First, one can replace \( \Phi^+ \) with a set \( \scR \) of integral vectors in \( \mathbb{Z}^n \), where subsets of size \( n \) satisfy the conditions in Lemma \ref{lem:-1,0,1}. For example, \( \scR \) could be the set of columns of a \emph{totally unimodular matrix}, where every square submatrix has determinant \( 0 \), \( 1 \), or \( -1 \). Notably, \( \Phi^+ \) can be viewed as a totally unimodular matrix whose columns satisfy the \emph{consecutive-ones property}, meaning that in the coefficient matrix representation of \( \Phi^+ \), the \( 1 \)s in each column appear consecutively.

Second, one can consider \( \Phi \) to be a root system of a different type. For instance, if \( \Phi \) is of type \( B_n \), then \( \det(S) \in \{-2, -1, 0, 1, 2\} \) for any subset \( S \subseteq \Phi^+ \) with \( \#S = n \). In this case, the signature would count the occurrences of \( 1 \)s, \( 2 \)s, \( -1 \)s, and \( -2 \)s among the cofactors. We leave the exploration of these generalizations for future research.
\end{remark}

 \begin{remark} 
\label{rem:spectral} 
Another direction for future research would be to explore whether there is a connection between the signature defined here and the concept of  \emph{inertia} in real symmetric matrices, which counts the number of positive, negative, and zero eigenvalues of a matrix. 
\end{remark}

We continue with a brief overview of the combinatorial setup related to   Eulerian numbers.
Let $\mathfrak{S}_{n}$ denote the symmetric group on $[n]$. 
For any permutation $\sigma\in\mathfrak{S}_{n}$, we define a \tbf{descent} to be a position  $i \in [n-1]$  such that $\sigma(i) > \sigma(i + 1)$, and we denote by $\dsc(\sigma)$ the number of descents of $\sigma$.
The \tbf{$k$-Eulerian number}  $\left\langle {n \atop k}\right\rangle$ for $0 \le k \le n-1$, is defined to be the number of permutations in $\mathfrak{S}_{n}$ with exactly $k$ descents, i.e.,
$$\left\langle {n \atop k}\right\rangle= \#\{\sigma\in\mathfrak{S}_{n} \mid \dsc(\sigma) = k\}.$$

We are now ready to present our main result on signatures.
\begin{theorem}
\label{thm:main}
Let $\Phi$ be a root system of type $A_n$.
For $1\le a \le b \le n+1$ with $a+b \le n+1$, define
$$ s_{a,b} := \#\left\{S \subseteq \Phi^+ \mid \#S=n+1 \mbox{ and } \sign(S)=\{a,b\}  \right\} $$
and
$$
u_{a,b} := (n+1)^{n+1-a-b} \binom{n}{a+b-1} \left\langle {a+b-1 \atop a-1}\right\rangle.
$$
Then
$$
s_{a,b} =
\begin{cases}
u_{a,b} & \text{if } a<b, \\
\frac{1}{2}u_{a,b} & \text{if } a=b \ne1,\\
0  & \text{if } a=b=1.
\end{cases}
$$
\end{theorem}

It is interesting to note that the classical Eulerian numbers appear in the counting formula for the signature. But what is the signature used for? 
We discuss two applications related to the \emph{cones} of the \emph{deformed arrangements} defined by $\Phi$ (Theorems \ref{thm:app2} and \ref{thm:app1}).

Let $N :=\#\Phi^+$ denote the number of positive roots in the root system $\Phi$.
Recall the notation $\Cca_{\Phi^+}= (c_{ij})  \in \Mat_{n\times N}(\Z)$ of the coefficient matrix of $\Phi^+$ with respect to $\Delta$. 
Let $E_1,\ldots, E_{N} \subseteq \Z$ be nonempty sets of integers. 
Define $\A^E_n$ as an $(n+1) \times \left( \sum_{j=1}^N \#E_j + 1 \right)$ integral matrix with columns
$$\begin{pmatrix}
c_{1j} \\ \vdots \\ c_{nj} \\\ -x
\end{pmatrix},
  \quad \begin{pmatrix}
0 \\ \vdots \\ 0 \\\ 1
\end{pmatrix} \quad \text{ for all } x \in E_j, \,1 \le j \le N.
$$ 
We call $\A^E_{n}$ the \tbf{$E$-deformation} of $\Phi$. 

A typical example of deformations is those that are ``uniform", meaning that all the sets $E_j$ are identical.
Let $\ell \le m$ be integers. 
Denote $[\ell,m] := \{ k \in \Z \mid \ell \le k \le m\}$. 
The \tbf{$[\ell,m]$-deformation} $\A_n^{[\ell,m]}$ of $\Phi$ is an $E$-deformation where $E_j =[\ell,m] $ for all $1 \le j \le N$. 
For example, if $\Phi=A_2$, then
$$
\A_2^{[0,1]} = 
\begin{pmatrix}
1 & 1 & 0 & 0 & 1 &   1 & 0\\
0& 0 & 1 & 1 & 1 &   1 & 0\\
-1 & 0 & -1& 0 & -1 &   0 & 1
\end{pmatrix}.
$$

A vector in a real vector space defines a linear hyperplane that is orthogonal to it.
 Thus, a matrix defines an arrangement of hyperplanes, one for each column (see \S\ref{sec:SP}). 
 
 The root system deformation defines several well-known arrangements in the literature. 
 Let $m \ge 1$ be an integer. 
 The (central) hyperplane arrangements defined by $[1-m,m]$, $[-m,m]$, and $[1,m]$-deformations are known as the  \tbf{cones} over the (extended)  \tbf{Shi}, \tbf{Catalan}, \tbf{Linial}  arrangements of $\Phi$, respectively. 
 When the context is clear, we use the notations $\cc\mathrm{Shi}_n^{[1-m,m]}$, $\cc\Cat_n^{[-m,m]}$, and $\cc\Lin_n^{[1,m]}$ to refer to these cones and their corresponding matrices. 
 For example, $\cc\Shi_2^{[0,1]}$ refers to both $\A_2^{[0,1]}$ and the corresponding Shi cone.

 We are also interested in a non-uniform deformation closely related to $\cc\Shi_n^{[0,1]}$, known as the \tbf{Ish arrangement} \cite{Arms13}, and its cone denoted $\cc\Ish_n$. 
 The construction of Ish modifies $\cc\Shi_n^{[0,1]}$ by replacing each column of the form $(\alpha_{i,j} , -1)^T$ with $(\alpha_{1,j} , -i)^T$, where $^T$ denotes the transpose of the matrix.
 For instance, $\cc\Ish_2$ differs from $\cc\Shi_2^{[0,1]}$ only in replacing $(0 , 1 , -1)^T$ with $(1, 1 , -2)^T$. 
 It is worth noting that $\Shi_n^{[0,1]}$ and $\Ish_n$ share many properties, such as their \emph{characteristic polynomial}, the number of regions of a given number of \emph{ceilings} and \emph{degrees of freedom}, and the \emph{freeness of the derivation module} \cite{Arms13, AR12, Atha98, AST17}.

Given an integral matrix $\A$, one of the most studied polynomials associated with $\A$ is the so-called \tbf{Tutte polynomial} and its arithmetic generalization. 
We use the notation $S \subseteq \A$ to indicate a submatrix $S$ containing columns of $\A$ (or simply a subset $S$ of $\A$ when viewed as a set of column vectors). 
For each $S \subseteq \A$, denote by $r(S)$ the rank of $S$. 
The Tutte polynomial $T_\A(x,y)$ of $\A$ is a bivariate polynomial defined by
$$T_{\A}(x, y):=
\sum_{S\subseteq \A}(x-1)^{r(\A)-r(S)}(y-1)^{\#S-r(S)}.$$ 

The  \tbf{arithmetic multiplicity} of $S$, denoted $e(S)$, is the product of all elementary divisors of $S$. 
In particular, if $S$ is a \tbf{base} of $\A$ (i.e., $r(S) = r(\A) = \#S$, in the matroid sense), then $e(S) = |\det(S)|$. 
The \tbf{arithmetic Tutte polynomial} $T_\A^{\arith}(x,y)$ of $\A$ is defined as
$$T_{\A}^{\arith}(x, y):=
\sum_{S\subseteq \A}e(S)\cdot(x-1)^{r(\A)-r(S)}(y-1)^{\#S-r(S)}.$$ 
While the classical Tutte polynomial was introduced to study hyperplane arrangements and matroids, the arithmetic Tutte polynomial was developed to study \emph{toric arrangements} and \emph{arithmetic matroids} \cite{L12, DM13, BM14}. 
These polynomials have several notable evaluations that capture algebraic, combinatorial, and topological information about the arrangements, including the evaluation at $(1,1)$. 
By definition, $T_\A(1,1)$ counts the number of bases of $\A$, while $T_\A^{\arith}(1,1)$ is the sum of the arithmetic multiplicities of all bases. 
The latter also has two other interpretations: The volume of the \emph{zonotope} and the dimension of the  \emph{Dahmen-Micchelli space} associated with $\A$ \cite{L12}.

Explicit calculations of the classical and arithmetic Tutte polynomials of classical root systems ($ABCD$) are known \cite{Ard07, ACH15}. 
However, for arbitrary deformations, these calculations become quite complicated due to the involvement of deformation parameters and the need for extensive calculations of multiplicities. Our first application of Theorem \ref{thm:main} provides explicit formulas for $T_\A(1,1)$ and $T_\A^{\arith}(1,1)$ when $\A$ is an $[\ell,m]$-deformation of a type $A$ root system.

\begin{notation} 
\label{not} 
Let $k$ be an integer and $n,d$ be positive integers. Define
\begin{align*}
 \delta_n & := (n+1)^{n-1},\\
\tau_{n,k,d} & :=
\sum_{i=0}^{ \lfloor  \frac{k-n}d \rfloor} (-1)^i {n \choose i} {k-id-1 \choose n-1} .
\end{align*}
Here $ \lfloor x\rfloor$ denotes the floor function. 
\end{notation}

\begin{theorem} 
\label{thm:app2} 
Let $\Phi$ be a root system of type $A_n$.
Let $\ell, m$ be integers such that $|\ell| \le m$. 
Denote  $d:=m-\ell+1>0$. 
Then
\begin{align*}
T_{\A_n^{[\ell,m]}}(1, 1) & = \delta_n + n\delta_n d^{n-1}\sum_{k=1}^{m-\ell}(d-k)+ \sum_{ {1\le a \le b \le n+1 \atop a+b \le n+1} } d^{n+1-a-b} s_{a,b} \sum_{k=1}^{mn-\ell}(\tau_{a+b,k_1,d}+\tau_{a+b,k_2,d}), \\
T_{\A_n^{[\ell,m]}}^{\arith}(1, 1) & = \delta_n + n\delta_n d^{n-1}\sum_{k=1}^{m-\ell}(d-k)k+ \sum_{ {1\le a \le b \le n+1 \atop a+b \le n+1} } d^{n+1-a-b} s_{a,b} \sum_{k=1}^{mn-\ell}(\tau_{a+b,k_1,d}+\tau_{a+b,k_2,d})k,
\end{align*}
where $k_1:= k+a(1-\ell)+b(m+1)$ and $k_2:= k+b(1-\ell)+a(m+1)$.

\end{theorem}

Another significant invariant associated with an integral matrix $\A$ is the \tbf{characteristic quasi-polynomial} $\chi^{\quasi}_\A(q)$, which enumerates the size of the complement of the arrangement defined by $\A$ modulo a positive integer $q$ (see \S\ref{sec:SP}). 
The characteristic quasi-polynomial is an important counting function that encodes both arithmetic and combinatorial properties of hyperplane and toric arrangements \cite{LTY21, TY19}. 
Determining the minimum (or best) period of a quasi-polynomial is generally a challenging problem. 
It has been proven that the minimum period $\rho\A$ of $\chi^{\quasi}_\A(q)$ equals the \tbf{lcm period} \cite{HTY23}, defined as the least common multiple of the largest elementary divisors of all submatrices consisting of the columns of $\A$. 
It is easily seen that a type $A$ root system has lcm period $1$. 
However, for its deformation, the lcm period is nontrivial and has not been widely studied.

We aim to initiate a study of this lcm period and show that the signature is an efficient tool for computing it in most of significant cases. 
This leads to our second application of Theorem \ref{thm:main}.

\begin{theorem} 
\label{thm:app1} 
Let $\Phi$ be a root system of type $A_n$.
Let $\ell, m$ be integers such that $|\ell| \le m$ and $m+1 \ge n\ell$. 
The minimum periods of the characteristic quasi-polynomials of the cones over the $[\ell,m]$-deformed and Ish arrangements are given by
\begin{align*}
\rho_{\A_n^{[\ell,m]}}  & = \lcm\{1,2,\ldots,mn-\ell\}, \\
\rho_{\cc\Ish_n} & =  \lcm\{1,2,\ldots, n\}.
\end{align*}
In particular,  $\cc\Shi_n^{[0,1]}=\A_n^{[0,1]}$ and $\cc\Ish_n $ share the same minimum period. 
\end{theorem}

Note that $\cc\mathrm{Shi}_n^{[1-m,m]}$ and $\cc\Cat_n^{[-m,m]}$ satisfy the conditions in  Theorem \ref{thm:app1}, while $\cc\Lin_n^{[1,m]}$ satisfies the second condition when $m+1 \ge n$.

 \section{Signatures and Eulerian numbers}
 \label{sec:SE}
 
 In this section, we provide the proof of our main result, Theorem \ref{thm:main}. 
 We begin by recalling some properties of Eulerian numbers.

 For any permutation $\sigma \in \mathfrak{S}_n$, we define an \tbf{ascent}  as a position $i \in [n-1]$ such that $\sigma(i) < \sigma(i+1)$, and we denote by $\asc(\sigma)$ the number of ascents of $\sigma$. It is straightforward to observe that the numbers of descents and ascents satisfy the identity:
$$\asc(\sigma)+\dsc(\sigma)=n-1.$$

 Moreover, reversal of a permutation swaps ascents and descents. 
 This provides a natural bijection between the set of permutations with $k$ descents and those with $k$ ascents. 
 In particular, Eulerian numbers are  \tbf{palindromic}, meaning:
$$ \left\langle {n \atop k}\right\rangle= \left\langle {n \atop n-1-k}\right\rangle.$$

Next, we introduce cyclic descents and cyclic ascents for permutations in $\mathfrak{S}_n$. 
For a permutation $\sigma \in \mathfrak{S}_n$, we define a  \tbf{cyclic descent}  as a position $i \in [n]$ such that $\sigma(i) > \sigma(i+1)$, where we set $\sigma(n+1) = \sigma(1)$. 
We denote by $\widetilde{\dsc}(\sigma)$ the number of cyclic descents of $\sigma$. 
Similarly, we define  \tbf{cyclic ascents} and denote the number of cyclic ascents by $\widetilde{\asc}(\sigma)$.

The \tbf{cyclic $k$-Eulerian number}, denoted by $\widetilde{\left\langle {n \atop k} \right\rangle}$ for $1 \leq k \leq n-1$, counts the number of permutations in $\mathfrak{S}_n$ with exactly $k$ cyclic descents, i.e.,
$$\widetilde{\left\langle {n \atop k}\right\rangle}= \#\{\sigma\in\mathfrak{S}_{n} \mid\widetilde{\dsc}(\sigma) = k\}.$$

The classical and cyclic Eulerian numbers are related by the following formula:

 \begin{proposition}[{\cite[Proposition 1.1]{Peter05}}]
 \label{prop:cc}
 For $1 \le k \le n-1$, we have:
$$
\widetilde{\left\langle {n \atop k}\right\rangle} =n
 \left\langle {n-1 \atop k-1}\right\rangle.
$$
\end{proposition}

We now turn to some fundamental properties of the signature.
 
 \begin{lemma}
 \label{lem:unchanged}
 
Let $\beta_1,\ldots,\beta_{n+1} \in \Phi^+$ be mutually distinct positive roots.
For any permutation $\sigma\in\mathfrak{S}_{n+1}$,   we have:
$$
\sign(\beta_1,\ldots,\beta_{n+1}) =
\sign(\beta_{\sigma(1)},\ldots,\beta_{\sigma(n+1)}).
$$
\end{lemma}

\begin{proof} 
Fix $1 \leq i \leq n$. 
Consider the two tuples:
$$
S = (\beta_1,\ldots,\beta_{i-1},\beta_{i},\beta_{i+1},\beta_{i+2},\ldots, \beta_{n+1}) \mbox{ and }
S'=(\beta_1,\ldots,\beta_{i-1},\beta_{i+1},\beta_{i},\beta_{i+2},\ldots, \beta_{n+1}) .
$$

Since the symmetric group on a finite set is generated by consecutive transpositions, it suffices to prove that $\sign(S) = \sign(S')$.
The cofactors of $S$ are denoted by $d_k$ (from Definition \ref{def:sign}), and it is easy to show that $d_i=-d'_{i+1}$, $d_{i+1}=-d'_{i}$ and $d_{k}=-d'_{k}$ for all $k \notin \{i,i+1\}$.
The assertion follows.
\end{proof}

Let \( \Phi \) be a root system of type \( A_n \). In the remainder of this section, we describe \( \Phi \) in terms of the basis \( \mathcal{E} = \{\epsilon_1, \ldots, \epsilon_{n+1}\} \) for \( V = \mathbb{R}^{n+1} \). The positive system \( \Phi^+ \) can be identified with the complete graph \( K_{n+1} \) with vertex set \( [n+1] \), where each positive root \( \epsilon_i - \epsilon_j \in \Phi^+ \) corresponds to the edge \( \{i, j\} \in K_{n+1} \). 
Thus, each subset \( S \subseteq \Phi^+ \) can be viewed as a subgraph of \( K_{n+1} \). A subset \( S \) of positive roots is linearly dependent if and only if it contains a cycle in \( K_{n+1} \).

We now need a few results from graph theory.

 \begin{lemma}[Circuit rank]
 \label{lem:cr}
 
 Let $G$ be an undirected graph, and let $v$, $e$, and $c$ denote the number of vertices, edges, and connected components of $G$, respectively. 
 The \tbf{circuit rank} $r$ of $G$, defined as the number of independent cycles in $G$, is given by:
$$r = e-v+c.$$
\end{lemma}

 \begin{lemma}
 \label{lem:conn}
If $S \subseteq K_{n+1}$ is a subgraph consisting of $n+1$ edges and contains exactly one cycle, then $S$ must be connected.\footnote{A connected graph containing exactly one cycle is called \emph{unicyclic}.}
\end{lemma}

\begin{proof} 
Let $v$ and $c$ denote the number of vertices and connected components of $S$, respectively. 
Since the circuit rank of $S$ is $1$ and $S$ has $n+1$ edges, by Lemma \ref{lem:cr}, we have:
$$v-c=n.$$
In particular, since $c \ge 1$, we must have $v \ge n+1$. 
However, since $S \subseteq  K_{n+1}$, we also have $v \leq n+1$. 
Hence, $v=n+1$ and $c=1$, which implies that $S$ is connected.
\end{proof}

The following is the key ingredient for the proof of Theorem  \ref{thm:main}.
 
  \begin{lemma}
 \label{lem:uni}
 
 Let $S \subseteq  K_{n+1}$ be a subgraph consisting of $n+1$ edges.
If  $S$ contains only one cycle, say $\Gamma= (v_1,\ldots,v_\ell, v_{\ell+1}=v_1)$ of length $\ell \le n+1$ given in terms of its vertex sequence,
then 
$$
\sign(S) =
\{\widetilde{\asc}(\sigma),\widetilde{\dsc}(\sigma)\},
$$
where $\sigma = v_1v_2\cdots v_\ell$ is a permutation in $ \mathfrak{S}_\Gamma$ ($\simeq \mathfrak{S}_{\ell}$) embedded in $\mathfrak{S}_{n+1}$.
\end{lemma}

\begin{proof} 
 For $1 \le i \le \ell$, define the following positive roots in $\Phi^+$: 
 $$
\beta_i :=
\begin{cases}
\epsilon_{v_i}-\epsilon_{v_{i+1}} & \text{if } v_i < v_{i+1}, \\
\epsilon_{v_{i+1}} - \epsilon_{v_i} & \text{if } v_i > v_{i+1}.
\end{cases}
$$
By Lemma \ref{lem:unchanged}, we may assume that $S$ is represented by the tuple $S=(\beta_1,\ldots,\beta_{\ell}, \gamma_{\ell+1},\ldots,\gamma_{n+1})$ for some positive roots $\gamma_j \in \Phi^+$.

 Since $\Gamma= (v_1,\ldots,v_\ell)$ is a cycle, we have the identity: 
  $$(\epsilon_{v_1}-\epsilon_{v_{2}} ) +\cdots +(\epsilon_{v_\ell}-\epsilon_{v_{\ell+1}}) =0,$$
which can be written as
$$b_1\beta_1 +\cdots +b_\ell\beta_\ell=0,$$
where 
 $$
b_i :=
\begin{cases}
1 & \text{if } v_i < v_{i+1}, \\
-1 & \text{if } v_i > v_{i+1}.
\end{cases}
$$
The number of cyclic ascents and descents in $\sigma$ is then given by:
\begin{align*}
\widetilde{\asc}(\sigma)  & =\#\{i \in [\ell] \mid b_i=1 \}, \\
\widetilde{\dsc}(\sigma)  & =\#\{i \in [\ell] \mid b_i=-1\}.
\end{align*}

Recall the notation $d_k$ for the cofactors of $S$ for $1 \le k \le n+1$ in Definition \ref{def:sign}. 
First, observe that since $\beta_1,\ldots,\beta_{\ell}$ are linearly dependent, we must have $d_k=0$ for all $\ell+1 \le k \le n+1$. 
Next, consider the linear dependence of the vectors $b_1\beta_1+b_2\beta_2$ and $\beta_3,\ldots, \beta_\ell$. 
Specifically, we have the following determinant relation:
$$0=\det(b_1\beta_1+b_2\beta_2,\beta_3,\ldots, \beta_\ell) = b_1\det(S_2)+ b_2 \det(S_1) = b_1d_2 -b_2d_1.$$
Therefore, $b_1=b_2$ if and only if $d_1=d_2$. 

By iterating this argument using the linear dependence of \( \beta_1, \ldots, \beta_{i-1}, b_i \beta_i + b_{i+1} \beta_{i+1}, \beta_{i+2}, \ldots, \beta_\ell \) for all \( 1 \leq i \leq \ell-1 \), we can deduce that \( b_i = b_{i+1} \) if and only if \( d_i = d_{i+1} \).
Consequently, we obtain:
$$
\sign(S) =
\{\widetilde{\asc}(\sigma),\widetilde{\dsc}(\sigma)\}.
$$
\end{proof}

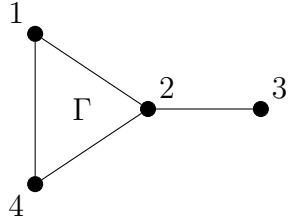
\begin{figure}[htbp]
\centering
\begin{tikzpicture}
\draw (0,0) node[v](1){} node[above right]{$2$};
\draw (1.5,0) node[v](2){} node[above right]{$3$};
\draw (-1.5,1) node[v](3){} node[above left]{$1$};
\draw (-1.5,-1) node[v](4){} node[below left]{$4$};
\draw (-.6,0) node[](5){} node[left]{$\Gamma$};
\draw (2)--(1)--(3)--(4)--(1);
\end{tikzpicture}
\caption{A graphical interpretation of the subset \( S \) in \( A_3 \) from Example \ref{ex:1st}. The graph corresponding to \( S \) is unicyclic, with the cycle \( \Gamma = (1,2,4,1) \) of length 3. The permutation \( \sigma = 124 \in \mathfrak{S}_\Gamma \) (which is isomorphic to \( \mathfrak{S}_3 \)) has \( \widetilde{\asc}(\sigma) = 2 \) and \( \widetilde{\dsc}(\sigma) = 1 \). According to Lemma \ref{lem:uni}, the signature of \( S \) is given by \( \sign(S) = \{1,2\} \), which aligns with the signature calculation presented in Example \ref{ex:1st}.}
\label{fig:graph-1st}
\end{figure}

See Figure \ref{fig:graph-1st} for an example of Lemma \ref{lem:uni}.
We also need the following fact. 
Let $\Part(X)$ denote the set of all partitions of a finite set $X$.
\begin{lemma} 
\label{lem:part} 
Let $n \in \Z_{\ge 0}$ be a nonnegative integer and $x$ be an intermediate variable. 
Then, the following identity holds:
$$\sum_{\lambda \in \Part([n])} x^{\#\lambda}\prod_{B \in \lambda} (\#B)^{\#B-1} = x(x+n)^{n-1}.$$

\end{lemma}
\begin{proof}
Let $P(n)$ denote the left-hand side of the equation. 
We proceed by an induction on $n$. 
The base case $n=0$ is obvious. 
Let  $\lambda'$ be a partition of $[n+1]$. 
Focus on the block of  $\lambda'$ that contains the element $n+1$. 
Suppose that there are $n-k$  elements, other than $n+1$, that belong to this block. 
We can choose these elements in  ${n \choose n-k}={n \choose k}$ ways. 
The remaining blocks of $\lambda'$, which consist of the remaining $(n+1) -(n-k+1)=k$ elements, form a partition of $[k]$. 
Therefore,
\begin{align*}
P(n+1) & = \sum_{k=0}^n {n \choose k} \sum_{\lambda \in \Part([k])} x^{1+\#\lambda}\prod_{B \in \lambda} (\#B)^{\#B-1}(n-k+1)^{n-k}  \\
& = x\sum_{k=0}^n {n \choose k}(n-k+1)^{n-k}\sum_{\lambda \in \Part([k])} x^{\#\lambda}\prod_{B \in \lambda} (\#B)^{\#B-1} \\
& = x\sum_{k=0}^n {n \choose k}(n-k+1)^{n-k}x (x+k)^{k-1} \\
& = x(x+n+1)^{n}.
\end{align*}
Note that we have applied the induction hypothesis in the third equality and used the Abel's binomial theorem in the last equality.
\end{proof}

\begin{proof}[\tbf{Proof of Theorem  \ref{thm:main}}]
Let $S \subseteq  K_{n+1}$ be a subgraph consisting of $n+1$ edges. 
In particular, $S$ cannot be a forest since $S$ is linearly dependent. 
For given  $1\le a \le b \le n+1$ with $a+b \le n+1$, we need to compute:
$$ s_{a,b}= \#\left\{S \subseteq  K_{n+1} \mid \#S=n+1 \mbox{ and } \sign(S)=\{a,b\}  \right\} .$$

If  $S$ contains more than one cycle, then $ \sign(S) =
\{0,0\}$ since $S \setminus \{\beta\}$ contains a cycle for any $\beta \in S$. 
Thus, we may assume that any subgraph $S$ we are considering contains exactly one cycle. 
Moreover, by Lemma \ref{lem:conn}, $S$ must be connected. 

Let $\ell=a+b$. 
By Lemma \ref{lem:uni}, we may assume that $S$ contains a unique cycle $\Gamma= (v_1,\ldots,v_\ell,$ $v_{\ell+1}=v_1)$ of length $3 \le \ell \le n+1$.
Denote by $\sigma = v_1v_2\cdots v_\ell$ a permutation in $ \mathfrak{S}_\Gamma$ ($\simeq \mathfrak{S}_{\ell}$) embedded in $\mathfrak{S}_{n+1}$. 
We have the following formula:
$$s_{a,b}=\frac{1}{2\ell} {n+1 \choose \ell} \#\left\{\sigma \in \mathfrak{S}_{\ell} \mid \widetilde{\dsc}(\sigma) \in \{a,b\}  \right\}\sum_{\lambda \in \Part([n+1-\ell])} \prod_{B \in \lambda} \ell\cdot(\#B)^{\#B-1}.$$

Let us explain why. 
By Lemma \ref{lem:uni}, we can transform the problem of counting subgraphs \( S \) into a problem of counting permutations \( \sigma \). However, we must take into account the potential repetition of terms.
There are ${n+1 \choose \ell}$ ways to choose the vertices of the cycle $\Gamma$. 
A circular shift of a permutation preserves cyclic descents (and cyclic ascents), and reversal of a permutation swaps cyclic descents and ascents. 
Hence, for a given cycle $\Gamma$, there are $2\ell$ distinct permutations  $\sigma$ of $\Gamma$.

Next, because $S$  is connected, we need to count the number of forests of vertex set 
$[n+1-\ell]$  where each component is attached to $\Gamma$. 
For a given partition 
$\lambda \in \Part([n+1-\ell])$, each block
 $B$  in the partition corresponds to a tree, and there are 
 $(\#B)^{\#B-2}$ labeled trees on 
$B$ by Cayley's formula. 
The cycle 
$\Gamma$  is connected to each tree by an edge 
$e$, where one endpoint is from 
$\Gamma$ and the other is from 
$B$. There are 
$\ell\cdot\#B$  ways to choose such an edge for each block 
$B\in\lambda$. 
See Figure \ref{fig:illus} for a graphical illustration.

By applying  Lemma \ref{lem:part} with $\ell$ and $n+1-\ell$ in place of $x$ and $n$, respectively we obtain: 
$$s_{a,b}=\frac{1}{2} {n+1 \choose \ell} (n+1)^{n-\ell} \#\left\{\sigma \in \mathfrak{S}_{\ell} \mid \widetilde{\dsc}(\sigma) \in \{a,b\}  \right\}.$$

If $a=b=1$, then $\ell=2$ and obviously $s_{1,1}=0$. If $a=b>1$, apply Proposition \ref{prop:cc} to obtain:
$$
s_{a,a} =\frac{1}{2} {n+1 \choose \ell} (n+1)^{n-\ell} \widetilde{\left\langle {\ell \atop a}\right\rangle} = \frac{1}{2} {n\choose \ell-1} (n+1)^{n+1-\ell}  \left\langle {\ell-1 \atop a-1}\right\rangle.
$$

If $a<b$, we compute: 
\begin{align*}
s_{a,b} & =\frac{1}{2} {n+1 \choose \ell} (n+1)^{n-\ell}\left( \widetilde{\left\langle {\ell \atop a}\right\rangle} +  \widetilde{\left\langle {\ell \atop b}\right\rangle} \right) \\
& =\frac{1}{2} {n+1 \choose \ell} (n+1)^{n-\ell}\left(\left\langle {\ell-1 \atop a-1}\right\rangle +  \left\langle {\ell-1 \atop b-1}\right\rangle\right) \\
& = {n\choose \ell-1} (n+1)^{n+1-\ell}  \left\langle {\ell-1 \atop a-1}\right\rangle.
\end{align*}
Note that we have applied the palindromicity of Eulerian numbers in the last equality. 
This completes the proof of the main theorem.
\end{proof}

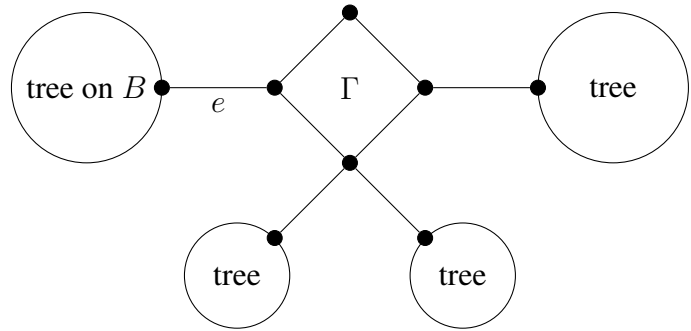
\begin{figure}[htbp!]
\centering
\begin{tikzpicture}
\draw (0,0) node[v](1){}  ;
\draw (1.5,0) node[v](2){} ;
\draw (-1,1) node[v](3){}  ;
\draw (-1,-1) node[v](4){};
\draw (-1,0) node[](5){} node[]{$\Gamma$};
\draw (-2,0) node[v](6){} ;

\draw (-3.5,0) node[v](11){} ;
    \draw (6) edge["$e$"] (11);

    \draw (-4.5,0) node[](12){} node[]{tree on $B$};
\draw (-4.5,0) circle (1);

\draw (-2.5,-2.5) node[](10){} node[]{tree};
\draw (-2.5,-2.5) circle (0.7);

\draw (2.5,0) node[](8){} node[]{tree};
\draw (2.5,0) circle (1);

\draw (0,-2) node[v](7){} ;

\draw (-2,-2) node[v](13){} ;
\draw (4)--(13);

\draw (.5,-2.5) node[](14){} node[]{tree};
\draw (.5,-2.5) circle (0.7);

\draw (4)--(7);

\draw (2)--(1)--(3)--(6)--(4)--(1);
\end{tikzpicture}
\caption{A graphical illustration of the proof of Theorem  \ref{thm:main}.}
\label{fig:illus}
\end{figure}

\begin{example}
If $n=6$, then
$$ (s_{a,b})_{1\le a\le b \le n+1, \, a+b\le n+1} =
\begin{pmatrix}
   0 & 36015 & 6860 & 735 & 42 & \:\:1 &   \\
     & 13720 & 8085 & 1092 & 57 &   &   \\
     &   & 1386 & 302 &   &   &
\end{pmatrix}.
$$
\end{example}

 \section{Signatures and bases}
 \label{sec:SB}
 
 In this section, we provide the proof of Theorem \ref{thm:app2}. 
 Let $\A$ be a matrix, and let $\mathsf{B}(\A)$ denote the set of all bases of $\A$. 
 Recall that the absolute value of the determinant of a base $B \in \mathsf{B}(\A)$ equals its arithmetic multiplicity, i.e., $e(B) = |\det(B)| > 0$. By definition: 
\begin{align*}
T_{\A} (1, 1) &  = \#\B(\A), \\
T_{\A}^{\arith}(1, 1) & = \sum_{B \in \B(\A)} e(B).
\end{align*}

Let $\A^E_n$ be an $E$-deformation of the root system $\Phi = A_n$. 
We aim to count the number of bases and compute the arithmetic multiplicity of each base of $\A^E_n$. 
Recall that a column in $\A^E_n$ is either $(0 \cdots 0 \, 1)^T$ or has the form 
$\begin{pmatrix} \beta \ -x \end{pmatrix}^T$ for some $\beta \in \Phi^+$ and $x \in E_j$. 
We define the submatrix of $\A^E_n$ containing the roots in $\Phi^+$ as the \textbf{root part} of $\A^E_n$.
Let $B$ be a base in $\A^E_n$. 
We note that $r(B) = \#B = r(\A^E_n) = n + 1$. We now consider the following cases:

\begin{enumerate}[{Case} 1.]
\item\label{C1} $B$ contains $(0\cdots 0 \, 1)^T$. In this case, $B$ can be written as:
$$
B = \begin{pmatrix}
X & \begin{matrix} 0 \\ \vdots \\ 0 \end{matrix} \\
\begin{matrix} -x_1 & \cdots & -x_n \end{matrix}  & 1
\end{pmatrix}
$$
for some subset $X \subseteq \Phi^+$. By Lemma \ref{lem:-1,0,1}, $e(B) = |\det(X)| = 1$. 
In particular, $X$ is linearly independent, and $\#X = n$.

\item\label{C2} $B$ does not contain $(0\cdots 0 \, 1)^T$ and $B$ has a duplicated column in the root part of $\A^E_n$. 
In this case, $B$ has the form:
$$
B = \begin{pmatrix}
Y & \beta &  \beta\\
\begin{matrix}-x_1&  \cdots &-x_{n-1} \end{matrix}  & -x_{n} & -x_{n+1}
\end{pmatrix}
$$
for some $\beta \in \Phi^+$ and $Y \subseteq \Phi^+$. 
Again by Lemma \ref{lem:-1,0,1}, we have:
$$e(B) = |\det(Y\, \beta)| \cdot |x_n - x_{n+1}|= |x_n - x_{n+1}|.$$ 
In particular, $Y \cup \{\beta\}$ is linearly independent and $\#(Y \cup \{\beta\})=n$.

\item\label{C3} $B$ does not contain $(0\cdots 0 \, 1)^T$   and $B$ does not have duplicated columns in the root part of $\A^E_n$. 
In this case, we define the signature of the root part of $B$. 
Assume the signature is $\{a,b\}$, where $1 \le a \le b \le n+1$ and $a + b \le n+1$. 
We also assume $a + b \ge 3$ since $s_{1,1} = 0$ by Theorem \ref{thm:main}. 
Applying Laplace expansion along the last row of $B$, we obtain:
$$e(B) = \left| \sum_{i=1}^a x_i - \sum_{i=a+1}^{a+b} x_i \right|.$$
\end{enumerate}

We will apply  the calculations above to  an $[\ell,m]$-deformation of $\Phi$. 
Before proceeding, we need some enumerative results.

\begin{lemma}
\label{lem:delta}

Let $\Phi = A_n$. 
The number of linearly independent subsets of $\Phi^+$ of cardinality $n$ is given by
$$\delta_n = (n+1)^{n-1}.$$
\end{lemma}

\begin{proof} 

Let $S \subseteq \Phi^+$ be an independent set with $\#S = n$. 
We have seen in \S\ref{sec:SE} that $S$ can be regarded as a subgraph of the complete graph $K_{n+1}$ on $n+1$ vertices. 
Since $S$ is linearly independent, it forms a forest. 
By Lemma \ref{lem:cr}, we have 
$$v - c = n,$$ where $v$ and $c$ are the numbers of vertices and connected components of $S$, respectively. 
Since $S \subseteq K_{n+1}$, we must have $v = n+1$ and $c = 1$. 
Therefore, $S$ is a connected tree, and the number of such subsets is equal to the number of labeled trees on $n+1$ vertices, which by Cayley's formula is:
$$\delta_n = (n+1)^{n-1}.$$

\end{proof}

\begin{lemma}[e.g., \cite{Mur81}] 
\label{lem:sol} 
Let $k$ be an integer and consider the linear equation in $n$ variables:
$x_1,x_2,\ldots,x_n$
$$x_1+x_2+\cdots+x_n = k.$$
Then the number of solutions to this equation in range $[d]$ for a fixed positive integer $d$ is given by
$$
\tau_{n,k,d} =
\sum_{i=0}^{ \lfloor  \frac{k-n}d \rfloor} (-1)^i {n \choose i} {k-id-1 \choose n-1} .
$$
In particular, the equation  has a solution when $n \le k \le nd$. 
\end{lemma}

\begin{corollary} 
\label{cor:sol} 
Let $\ell, m, a, b, k$ be integers such that $\ell \le m$ and $a,b,k \ge 1$. 
Define $k_1 := k + a(1-\ell) + b(m+1)$, $k_2 := k + b(1-\ell) + a(m+1)$, and $d := m - \ell + 1$. Consider the equation:
$$\left| \sum_{i=1}^a x_i - \sum_{i=a+1}^{a+b} x_i \right|= k.$$
Then the number of integer solutions to this equation in $[\ell, m]$ is:
$$\tau_{a+b,k_1,d}+\tau_{a+b,k_2,d}.$$
In particular, the equation has a solution when either (i) $a\ell-bm \le k \le am-b\ell$,  or (ii) $ b\ell-am\le k \le bm-a\ell$. 
\end{corollary}
\begin{proof}
First, consider the equation:
\begin{equation}
\label{eq:1} \sum_{i=1}^a x_i - \sum_{i=a+1}^{a+b} x_i= k.
\end{equation}
Define $y_i := x_i+1-\ell$ for $1 \le i \le a$ and  $y_i := -x_i+m+1$ for $a+1 \le i \le a+b$. 
Then $1 \le y_i \le d$ for all $i$. 
Moreover, Equation \eqref{eq:1} over  $[\ell, m]$ is equivalent to the following equation over  $[d]$:
 $$y_1+y_2+\cdots+y_{a+b} = k_1.$$
 Therefore, by Lemma \ref{lem:sol}, the number of solutions to Equation \eqref{eq:1} in   $[\ell, m]$ is $\tau_{a+b,k_1,d}$. 
Similarly, consider the equation:
\begin{equation}
\label{eq:2} \sum_{i=a+1}^{a+b} x_i - \sum_{i=1}^a x_i = k.
\end{equation}
This equation has $\tau_{a+b,k_2,d}$ solutions in   $[\ell, m]$. 
Since $k \ne 0$,  Equations \eqref{eq:1} and \eqref{eq:2} have no solutions in common. 
The main assertion follows. 

If  (i) $a\ell-bm \le k \le am-b\ell$ occurs, then $ a+b \le k_1 \le (a+b)d$. 
If  (ii) $a\ell-bm \le k \le am-b\ell$ occurs, then $ a+b \le k_2 \le (a+b)d$. 
By Lemma \ref{lem:sol}, the equation has a solution if either of the above cases occurs.
\end{proof}

\begin{proof}[\tbf{Proof of Theorem  \ref{thm:app2}}]

For Case \ref{C1}, by Lemma \ref{lem:delta}, there are $\delta_n$ ways to choose matrix $X$. Thus, this case contributes $\delta_n$ to both the classical and arithmetic Tutte polynomials.

For Case \ref{C2}, by Lemma \ref{lem:delta}, there are $\delta_n$ ways to choose matrix $(Y\, \beta)$. 
There are $n$ ways to choose $\beta$ and $d^{n-1}$ ways to choose the numbers $x_1, \ldots, x_{n-1}$. 
Note that $e(B)$ takes all and only values in $[m-\ell]$. 
If $e(B) = k$ for some $1 \le k \le m-\ell$, there are $d - k$ ways to choose an unordered pair $\{x_n, x_{n+1}\}$ from $[\ell, m]$ such that $|x_n - x_{n+1}| = k$. 
This explains the second term in the formulas.

For Case \ref{C3}, by Theorem \ref{thm:main}, there are $s_{a,b}$ ways to choose the root part of $B$. 
There are $d^{n+1-a-b}$ ways to choose the numbers $x_i$s other than $x_1, x_2, \ldots, x_{a+b}$. 
We can show that $e(B)$ takes values in $[mn - \ell]$  (but not necessarily all). 
This is true because the following inequalities hold
$$\ell-mn \le a\ell-bm \le \sum_{i=1}^a x_i - \sum_{i=a+1}^{a+b} x_i \le am-b\ell \le mn-\ell$$
under the conditions that $|\ell| \le m$, $1\le a \le b \le n+1$ and $a+b \le n+1$. 
Let us show, for example, the rightmost inequality 
$$am-b\ell \le (n+1-b)m-b\ell = mn + (1-b)(m+\ell) -\ell \le mn-\ell.$$
If $e(B) = k$ for some $1 \le k \le mn - \ell$, by Corollary \ref{cor:sol}, there are $\tau_{a+b,k_1,d} + \tau_{a+b,k_2,d}$ ways to choose a tuple $(x_1, x_2, \ldots, x_{a+b})$ from $[\ell, m]$ such that:
$$ \left| \sum_{i=1}^a x_i - \sum_{i=a+1}^{a+b} x_i \right| = k.$$
This explains the third term in the formulas.

\end{proof}

 \begin{remark}
 \label{rem:C23}
 In Case \ref{C2}, we only need to count the number of unordered pairs $\{x_n , x_{n+1}\}$  since the root $\beta$ is duplicated. 
 In Case \ref{C3}, we need to count all possible tuples $(x_1, \ldots, x_{a+b})$ since the roots in the root part are mutually distinct.
 
\end{remark}

 \section{Signatures and periods}
 \label{sec:SP}
 
 A function \( \varphi: \Z \to \C \) is called a \textbf{quasi-polynomial} if there exist a positive integer \( \rho \in \Z_{>0} \) and polynomials \( f^k(t) \in \Q[t] \) for \( 1 \le k \le \rho \) such that for any \( q \in \Z_{>0} \) with \( q \equiv k \pmod{\rho} \), we have:

\[
\varphi(q) = f^k(q).
\]

The number \( \rho \) is called a \textbf{period}, and the polynomial \( f^k(t) \) is called the \textbf{\( k \)-constituent} of the quasi-polynomial \( \varphi \). The smallest such \( \rho \) is called the \textbf{minimum period} of the quasi-polynomial \( \varphi \). The minimum period is necessarily a divisor of any period.

Let \( n, p \in \Z_{>0} \) be positive integers. Let \( \Cca = (c_1, \dots, c_p) \in \Mat_{n \times p}(\Z) \) be an integral matrix with all nonzero columns \( c_j \) for \( 1 \le j \le p \), and let \( b = (b_1, \dots, b_p) \in \Z^p \). Consider the augmented matrix:
\[
\A := \begin{pmatrix} \Cca \\ \hline b \end{pmatrix} \in \Mat_{(n + 1) \times p}(\Z).
\]

We call \( \Cca \) the \textbf{homogeneous part} of \( \A \). The matrix \( \A \) defines the following \textbf{integral} hyperplane arrangement in \( \R^n \):
\[
\As = \As(\A) := \{ H_j \mid 1 \le j \le n \},
\]
where
\[
H_j = H_{c_j} := \{ x \in \R^n \mid x c_j = b_j \}.
\]

We call the case \( b = 0 \) the \textbf{central case}, and the case \( b \neq 0 \) the \textbf{non-central case}. In the central case, we usually omit the vector \( b \) and identify \( \A = \Cca \). The horizontal line separating \( \Cca \) and \( b \) is used to distinguish the central and non-central cases.

In what follows, we assume that \( \Cca \) is \textbf{primitive}, i.e., each column \( c_j = (c_{1j} \dots c_{nj})^T \) of \( \Cca \) is primitive, meaning that \( \gcd\{c_{1j}, \dots, c_{nj}\} = 1 \). In particular, \( \A \) is also primitive. Thus, an integral arrangement \( \As \) is associated with a unique matrix \( \A \).

The \textbf{cone} \( \cc \As \) of \( \As \) is a \textit{central} arrangement in \( \R^{n+1} \) defined by
\[
\cc \As := \{ \cc H_j \mid 1 \le j \le n \} \cup {H_{\infty}},
\]
where
\[
\cc H_j := \left\{ x \in \R^{n+1} \mid x \begin{pmatrix} c_j \\ -b_j \end{pmatrix} = 0 \right\}, \quad H_{\infty} := \{ x \in \R^{n+1} \mid x_{n+1} = 0 \}.
\]

In other words, \( \cc \As \) is the central integral arrangement defined by the matrix (without the horizontal separating line)

\[
\A^E := \begin{pmatrix} \Cca & \begin{matrix} 0 \\ \vdots \\ 0 \end{matrix} \\ -b & 1 \end{pmatrix}.
\]

Let \( q \in \Z_{>0} \) and denote \( \Z_q := \Z / q \Z \). For \( a \in \Z \), let \( [a]_q := a + q \Z \in \Z_q \) denote the \( q \)-reduction of \( a \). For a matrix \( A' \) with integral entries, denote by \( [A']_q \) the entry-wise \( q \)-reduction of \( A' \). The \textbf{\( q \)-reduction} \( \As_q \) of \( \As \) is defined by
\[
\As_q = \As_q(\A) := \{ H_{j,q} \mid 1 \le j \le n \},
\]
where
\[
H_{j,q} := \{ z \in \Z_q^n \mid z [c_j]_q = [b_j]_q \}.
\]

Denote \( \Z_q^{\times} := \Z_q \setminus \{0\} \). The complement \( \M(\As_q) \) of \( \As_q \) is defined by

\[
\M(\As_q) := \Z_q^n \setminus \bigcup_{j=1}^p H_{j,q} = \{ z \in \Z_q^n \mid z [c_j]_q - [b_j]_q \in (\Z_q^{\times})^n \}.
\]

For $\emptyset \ne J \subseteq [p]$, let $\Cca_J \in \Mat_{n\times \#J}(\Z)$ denote the submatrix of $\Cca$ consisting of the columns indexed by $J$.
Let $r(J):=r(\Cca_J)$ denote the rank of $\Cca_J$.
Let $0<e_{J,1} \mid e_{J,2} \mid \cdots \mid e_{J,r(J)}$ be the elementary divisors of $\Cca_J$.
The \tbf{lcm period} of $\Cca$ is defined by
$$\rho_\Cca:= \lcm \{ e_{J,r(J)} \mid \emptyset \ne J \subseteq [p] \}.$$
 
\begin{theorem}[{\cite[Theorem 3.1]{KTT11}}]
\label{thm:KTT}
There exists a number $q_0 = q_0(\A) \in \Z_{>0}$ depending on the matrix $\A$ and a monic quasi-polynomial $\chi^{\quasi}_{\A}(q)$ with period $\rho_{\Cca}$ such that $\#\M(\As_q)  = \chi^{\quasi}_{\A}(q)$ for all $q>q_0$. 
This quasi-polynomial is called the \tbf{characteristic quasi-polynomial} of $\A$ (or of $\As$). 

\end{theorem}

The name ``characteristic quasi-polynomial" is inspired by the fact that $\chi^{\quasi}_{\A}(q)$ is a generalization of the \tbf{characteristic polynomial}  $\chi_\As( t)$ (e.g., \cite[Definition 2.52]{OT92}) of $\As$.
\begin{theorem}[{e.g., \cite[Remark 3.3]{KTT11}}]
\label{thm:KTT11} 
Let \( f_\A^k(t) \in \Z[t] \) for \( 1 \le k \le \rho_\Cca \) denote the \( k \)-constituent of \( \chi^{\quasi}_{\A}(q) \). The following holds:
$$f^1_{\A}(t)=\chi_{\As}(t).$$
\end{theorem}

\begin{theorem}[{\cite[Theorem 1.2]{HTY23}}]
\label{thm:HTY} 
The minimum period of the characteristic quasi-polynomial of a central integral arrangement equals its lcm period.
\end{theorem}

 \begin{remark} 
\label{rem:C}
We have seen that the lcm period of the characteristic quasi-polynomial in the central or non-central case depends only on the homogeneous part. The characteristic polynomial of a non-central arrangement \( \As \) and that of its cone \( \cc \As \) are known to be related by the simple formula:
\[
\chi_{\cc \As}(t) = (t-1) \chi_{\As}(t).
\]
However, their characteristic quasi-polynomials or minimum periods are very different in general. For example, a type \( A \) root system has minimum period 1, while an arbitrary deformation of the root system may have a very large minimum period, as Theorem \ref{thm:app1} illustrates.
\end{remark}

The following upper bound for the lcm period will be useful later. For \( \emptyset \neq J \subseteq [p] \), recall that the arithmetic multiplicity \( e(\Cca_J) \) of \( \Cca_J \) is defined by
\[
e(\Cca_J) = \prod_{j=1}^{r(J)} e_{J,j},
\]
the product of the elementary divisors of \( \Cca_J \).

\begin{proposition}[{\cite[Remark 2.3]{HTY23}, \cite[Lemma 5.2]{BM14},}] 
\label{prop:ub}
Let $\Cca \in \Mat_{n\times p}(\Z)$ be an integral matrix with all nonzero columns. Then 
$$\mu_\Cca  := \lcm \{ e(\Cca_J) \mid\emptyset \ne J \subseteq [p],\, \Cca_J \text{ is a base of } \Cca \}.$$
is a period of $\chi^{\quasi}_{\Cca}(q)$. In particular, $\rho_{\Cca} \le\mu_\Cca$.
\end{proposition}

\begin{proof}[\tbf{Proof of Theorem  \ref{thm:app1}}]
 By Theorem \ref{thm:HTY}, it suffices to compute the lcm periods of these arrangements. First, we prove the assertion for \( \A_n^{[\ell,m]} \). Let us focus on Case \ref{C3} in the proof of Theorem \ref{thm:app2}. Recall that if \( B \in \B(\A_n^{[\ell,m]}) \) is a base, then \( e(B) \) takes values in \( [mn-\ell] \) but not necessarily all of them.

 Let \( B' \in \B(\A_n^{[\ell,m]}) \) be a base such that its root part has signature \( \{1,n\} \). In fact, there is only one subset of the root system with signature \( \{1,n\} \) since \( s_{1,n} = 1 \) by Theorem \ref{thm:main}. It is not hard to see that this subset consists of all the simple roots and the highest root. Thus, the base \( B' \) is given by
$$
B'=
\begin{pmatrix}
    1 & \cdots & 0&  1 \\
    \vdots & \ddots & \vdots    & \vdots \\
   0 & \cdots & 1 &   1 \\
      -x_1 & \cdots & -x_n & -x_{n+1} 
             \end{pmatrix}.$$
             for some $x_1,\ldots,x_n, x_{n+1} \in [\ell,m]$.

             Note that \( e(B') = |x_1 + \cdots + x_n - x_{n+1}| \). Moreover, by Corollary \ref{cor:sol} together with the condition \( m+1 \ge n\ell \), the equation \( |x_1 + \cdots + x_n - x_{n+1}| = k \) for \( 1 \le k \le mn-\ell \) always has a solution. Therefore, \( e(B') \) takes all and only values in \( [mn-\ell] \). In particular,
 $$ \{ e(B) \mid  B \in \B( \A_n^{[\ell,m]}) \} =[mn-\ell].$$
Hence, the period defined in Proposition \ref{prop:ub} is given by
$$\mu_{\A_n^{[\ell,m]}} = \lcm\{1,2,\ldots,mn-\ell\}.$$

Furthermore, using row reduction, one can show that the base \( B' \) has elementary divisors $1, \dots,$ $1,$ $e(B')$. This implies that
 $$\rho_{\A_n^{[\ell,m]}}  \ge \lcm\{1,2,\ldots,mn-\ell\}.$$
 By Proposition \ref{prop:ub}, we must have 
 $$\rho_{\A_n^{[\ell,m]}}  = \lcm\{1,2,\ldots,mn-\ell\}.$$

Next, we prove the assertion for \( \cc \Ish_n \). Recall that
$$
 \cc\Ish_n= \begin{pmatrix}
\Phi^+ & \alpha_{1,j} &\begin{matrix} 0 \\ \vdots \\ 0 \end{matrix} \\
0 & -x & 1
\end{pmatrix},
$$
where $\alpha_{1,j}=\epsilon_1-\epsilon_j =  \sum_{k=1}^{j-1} \alpha_k$ for $2 \le  j \le n+1$ and $1 \le x \le j-1$.
Identify the positive system $\Phi^+$ of $\Phi$ with the edges of the complete graph $K_{n+1}$ with vertex set $[n+1]$ as in \S\ref{sec:SE}. 

Let $B \in \B(  \cc\Ish_n)$ be a base that satisfies Case \ref{C3}, i.e., $B$ does not contain $(0\cdots 0 \, 1)^T$ and the root part $S$ of $B$ does not have duplicate columns. 
Denote $\sign(S) = \{a,b\}$ for $1\le a \le b \le n+1$ with $3 \le a+b \le n+1$. 
By Lemma \ref{lem:uni} (and the beginning of the proof of Theorem  \ref{thm:main}), $S$ contains a unique cycle $\Gamma= (v_1,\ldots,v_{a+b}, v_{a+b+1}=v_1)$ of length $ a+b $ and 
$$
\sign(S) =
\{\widetilde{\asc}(\sigma),\widetilde{\dsc}(\sigma)\},
$$
where $\sigma = v_1v_2\cdots v_{a+b} \in \mathfrak{S}_{a+b}$. 

Therefore,
$$e(B) = \left| \sum_{i=1}^a x_i - \sum_{i=a+1}^{a+b} x_i \right|,$$
where $0 \le x_1,\ldots,x_{a+b} \le n$ are the parameters assigned to the edges of $\Gamma$. 
Note that the root $\alpha_{1,j} =\epsilon_1-\epsilon_j$ corresponds to the edge $\{1,j\}$. 
Thus, $S$ cannot have more than two roots of the form $\alpha_{1,j}$ since  $\Gamma$ is a cycle that cannot have more than two edges of the form $\{1,j\}$. 

If $S$ has only one root of the form $\alpha_{1,j}$, then $e(B)=|x|$ for some $0 \le x \le n$. 
If $S$ has two roots of that form, say $\alpha_{1,j_1}$ and $\alpha_{1,j_2}$, then $\Gamma= (1,j_1,j_2)$. 
Therefore, $\sign(S) = \{1,2\}$ and $e(B)=|x_{j_1}-x_{j_2}|$ for some $0 \le x_{j_1}, x_{j_2} \le n$. 
By combining these results with Cases \ref{C1} and \ref{C2}, we deduce that 
 $$ \{ e(B) \mid  B \in \B( \cc\Ish_n) \} =[n].$$
Hence, the period defined in Proposition \ref{prop:ub} is given by 
 $$\mu_{\cc\Ish_n}  = \lcm\{1,2,\ldots,n\}.$$
 
Now, consider a base \( B'' \in \B(\cc \Ish_n) \) of the form
  $$
B''=
\begin{pmatrix}
    1 & \cdots & 0&  1 \\
    \vdots & \ddots & \vdots    & \vdots \\
   0 & \cdots & 1 &   1 \\
      0 & \cdots & 0 & -x 
             \end{pmatrix}$$
             for some $1 \le x \le n$. 
             It is easily seen that $B''$ has elementary divisors $1,\ldots,1,e(B'')=|x|$. 
 This implies that  
 $$\rho_{ \cc\Ish_n}  \ge \lcm\{1,2,\ldots,n\}.$$
 By Proposition \ref{prop:ub}, we conclude that 
 $$\rho_{\cc\Ish_n}  = \lcm\{1,2,\ldots,n\}.$$
\end{proof}

\begin{remark} 
\label{rem:SI}
By Theorem \ref{thm:app1}, \( \cc\Shi_n^{[0,1]} \) and \( \cc\Ish_n \) share the same minimum period for every \( n \). However, they have different characteristic quasi-polynomials. Specifically, one may compute
\[
\chi^{\quasi}_{\cc\Shi_2^{[0,1]}}(q) =
\begin{cases}
(q-1)(q-3)^2 & \text{if } q \equiv 1 \pmod{2}, \\
(q-2)(q^2-5q+5) & \text{if } q \equiv 2 \pmod{2}.
\end{cases}
\]
and
\[
\chi^{\quasi}_{\cc\Ish_2}(q) =
\begin{cases}
(q-1)(q-3)^2 & \text{if } q \equiv 1 \pmod{2}, \\
(q-2)(q-3)^2 & \text{if } q \equiv 2 \pmod{2}.
\end{cases}
\]
\end{remark}

\vskip 1em
\noindent
\textbf{Acknowledgments.} 
The authors thank Maddalena Pismataro for helpful discussions.

\bibliographystyle{abbrv}

\end{document}